\documentclass[10pt,a4paper]{amsart}
\usepackage[latin1]{inputenc}
\usepackage[english]{babel}
\usepackage{geometry}
\usepackage{enumerate}
\usepackage{amsmath, amsfonts, amscd, amssymb, amsthm , array}
\usepackage{latexsym}
 \usepackage{graphicx}

\usepackage[T1]{fontenc}
\renewcommand{\P}{\mathbb{P}}
\newcommand{\M}{\mathcal{M}}
\newcommand{\A}{\mathcal{A}}
\newcommand{\B}{\mathcal{B}}
\renewcommand{\t}{\tau}
\newcommand{\D}{\Delta}

\renewcommand{\l}{\lambda}
\renewcommand{\L}{\Lambda}
\newcommand{\G}{\Gamma}

\renewcommand{\to}{\mapsto}

\DeclareMathOperator{\rk}{rk}
\DeclareMathOperator{\Sym}{Sym}

\DeclareMathOperator{\Proj}{Proj}

\renewcommand{\o}{\omega}

\newcommand{\h}{Hyp}
\newcommand{\RR}{\mathcal{R}}
\renewcommand{\H}{\mathcal{H}}

\newcommand{\C}{\mathbb{C}}

\newcommand{\m}{\mathfrak{m}}
\renewcommand{\to}{\rightarrow}

\theoremstyle{plain}
\newtheorem{theorem}{Theorem}[section]
\newtheorem{corollary}[theorem]{Corollary}
\newtheorem{definition}[theorem]{Definition}
\newtheorem{proposition}[theorem]{Proposition}
\newtheorem{criterion}[theorem]{Criterion}

\theoremstyle{definition}
\newtheorem{remark}[theorem]{Remark}
\newtheorem{lemma}[theorem]{Lemma}
\newtheorem{example}[theorem]{Example}

\title{Hyperelliptic Schottky problem\\
 and stable modular forms}
\date{}
\author[G.Codogni]{Giulio Codogni}
\address{Dipartimento di Matematica e Fisica, Universit\`{a} degli Studi Roma Tre, Largo S. Leonardo Murialdo 1, 00146 Rome (Italy)}
\email{codogni@mat.uniroma3.it}
\begin{document}
\begin{abstract}
It is well known that, fixed an even, unimodular, positive definite quadratic form, one can construct a modular form in each genus; this form is called the theta series associated to the quadratic form. Varying the quadratic form, one obtains the ring of stable modular forms. We show that the differences of theta series associated to specific pairs of quadratic forms vanish on the locus of hyperelliptic Jacobians in each genus. In our examples, the quadratic forms have rank 24, 32 and 48. The proof relies on a geometric result about the boundary of the Satake compactification of the hyperelliptic locus. We also study the monoid formed by the moduli space of all principally polarised abelian varieties, the operation being the product of abelian varieties. We use this construction to show that the ideal of stable modular forms vanishing on the hyperelliptic locus in each genus is generated by differences of theta series.

\end{abstract}
\maketitle
\begin{section}{Introduction}

The hyperelliptic Schottky problem is to characterise the locus of Jacobians of hyperelliptic curves inside the moduli space of principally polarised abelian varieties. A classical approach is to look for modular forms vanishing along the hyperelliptic locus; in other words, one looks for the equations of the hyperelliptic locus inside the moduli space of principally polarised abelian varieties. 

A special kind of modular forms are the theta constants; these were used by Mumford to give a solution to the hyperelliptic Schottky problem, as reviewed in Theorem \ref{SM}.

In this paper, we deal with stable modular forms. One nice feature of these forms is that they relate the theory of moduli spaces with the theory of quadratic forms.

To start with, let us explain what we mean by stable modular forms. Let $\A_g$ be the moduli space of principally polarised abelian $g$-fold defined over the field of complex numbers.  We consider the Satake compactification $\A_g^S$ of $\A_g$. This comes with a stratification
$$ \A_g^S=\A_g \sqcup \A_{g-1} \cdots \A_1 \sqcup \A_0  $$
In particular, we have a closed embedding
$$ \iota_g \colon \A_{g-1}^S \hookrightarrow \A_g^S $$
The collection of the moduli spaces $\A_g^S$ and these maps form a direct system of varieties; we can thus consider the ind-scheme
$$ \A_{\infty}:=\lim_g A_g^S $$
The basic definitions about ind-schemes are recalled in Section \ref{app}

Stable modular forms are naturally defined on $\A_{\infty}$. A stable modular form $F$ is a collection of modular forms $(F_g)_{g\geq 0}$: each $F_g$ is modular form on $\A_g^S$ and $\iota_g^*F_g=F_{g-1}$. We will recall the theory in Section \ref{satake}, in particular see Definition \ref{StableModForm}. A classical and surprising fact is that we can construct a stable modular form out of an even, unimodular, positive definite quadratic form $Q$. This stable modular form is called theta series associated to $Q$, and is denoted by $\Theta_Q$; see Definition \ref{thetaseries}. In particular, for any $g$, $\Theta_{Q,g}$ is a modular form on $\A_g^S$. In \cite{stabile}, Freitag showed that all stable modular forms are linear combinations of theta series.

In this set up, we can consider the ideal of stable modular forms vanishing on the locus $\h_g$ of hyperelliptic Jacobians in every genus. Let us formalise this with definition.
\begin{definition}[Stable Equation]
A stable equation for the hyperelliptic locus is a stable modular form $(F_g)_{g\geq0}$ such that $F_g$ vanishes along the hyperelliptic locus $\h_g$ for every $g$.
\end{definition}
Our first result is the following:

\begin{theorem}[= Theorem \ref{differences}]\label{Main_Intro}
The ideal of stable equations of the hyperelliptic locus is generated by differences of theta series
$$\Theta_P-\Theta_Q$$
where $P$ and $Q$ are even, unimodular, positive definite quadratic forms of the same rank.
\end{theorem}

A key ingredient in the proof of this results is a natural monoidal structure that one can put on $\A_{\infty}$. Given two principally polarised abelian varieties, their product is still a principally polarised abelian variety but of higher dimension. This defines an operation
$$
m\colon \A_{\infty}\times \A_{\infty}\to \A_{\infty}
$$
The pull-back $m^*$ gives to the ring of stable modular forms the structure of a co-commutative co-algebra. Because of this, we can run the general machinery explained in Section \ref{app} to prove Theorem \ref{Main_Intro}.

\smallskip

So far, the ideal of stable equations for the hyperelliptic locus could be trivial. Indeed, this is the case for the moduli space of curves: in \cite{CSB}, it is shown that the ideal of stable modular form vanishing on the Jacobian locus in any genus is trivial. In other words, given a non-zero stable modular form $F$, there exists a $g$ such that $F_g$ does not vanish on the moduli space of genus $g$ curves. In \cite{SB}, it is similarly shown that the ideal of stable equation for the $n$-gonal locus, with $n\geq 3$, is trivial. However, as we are going to see, the ideal of stable equations for the locus of Jacobians of hyperelliptic curves is far from being trivial.

\smallskip

The first stable equation for the hyperelliptic locus was discovered by C. Poor (\cite{Poor}): it is the difference of the theta series associated to the quadratic forms $D_{16}^+$ and $E_8\oplus E_8$; this modular form is also called the called Schottky form. To construct new stable equations, we need to know more about the geometry of the Satake compactification. The Satake compactification $\h_g^S$ of $\h_g$ will be defined in Section \ref{Hyperelliptic}. We denote by $\A_g^{ind}$ the moduli space of indecomposable principally polarised abelian $g$-fold.

\begin{theorem}[= Theorem \ref{Transversality}; Transversality]\label{Transversality_Intro}
The intersection of the Stake compactification $\h_{g+1}^S$ and $\A_g^{ind}$ inside $\A_{g+1}^S$ is scheme theoretically equal to $\h_g$.
\end{theorem}
The statement was well-known at the level of sets; we call it a transversality result because it states that the scheme structure of the intersection is the reduced one. The analogue result does not hold for the moduli of curves \cite[Theorem 1.1]{CSB} and for the moduli of $n$-gonal curves \cite{SB}, with $n\geq 3$. In those cases, the failure of the transversality implies that there are no stable equations; in the hyperelliptic case, this transversality result is key in the construction of stable equations.

Combining Theorem \ref{Transversality_Intro} and Criterion \ref{crit_slope} we can prove the following
\begin{theorem}[= Corollary \ref{first_equations} and Theorem \ref{second_equations}]\label{equations_intro}
The difference of theta series
$$\Theta_P-\Theta_Q$$
is a stable equation for the hyperelliptic locus when one of the following hold:
\begin{enumerate}
\item $\rk(P)=\rk(Q)=24$ and the two quadratic forms have same number of vectors of norm $2$;
\item $\rk(P)=\rk(Q)=32$ and the two  quadratic forms do not have any vector of norm $2$;
\item $\rk(P)=\rk(Q)=48$ and the two quadratic forms do not have any vector of norm $2$ or $4$;
\end{enumerate}

\end{theorem}

Each item of Corollary \ref{equations_intro} concerns a finite positive number of pairs of quadratic forms. In \cite{King}, it is shown that there are more than ten millions of quadratic forms meeting the hypothesis of the second item. In the first item, the "slope" of the quadratic form, i.e. the ratio between the rank and the norm of the shortest  non-zero vector, is strictly bigger than the slope of the hyperelliptic locus; because of this, in the proof we need to use some non-trivial arithmetic properties of the quadratic forms: namely we use Theorem \ref{polLat} via Corollary \ref{he}.

\smallskip

We think that the ideal of stable equations defines scheme theoretically the hyperelliptic locus inside the moduli space of indecomposable principally polarised abelian varieties in any genus. This, in particular, would imply that there are infinitely many pairs of quadratic forms which give stable equations for the hyperelliptic locus. In order to give a characterisation of these pairs, we think one should relate theta series to partition functions, as partially suggested in \cite{Volp}, \cite{Volp2} and \cite{Matone}. 

\begin{subsection}*{Acknowledgement} This work is part of my Ph.D. thesis \cite{PhD}, I would like to thank my advisor Nick Shepherd-Barron, without whom this work would not have been possible. It is a pleasure to thank Riccardo Salvati Manni for many useful advices. I also had the benefit of conversations with Enrico Arbarello, Gavril Farkas, Sam Grushevsky, Marco Matone, Sara Perna, Filippo Viviani and Roberto Volpato on the topics of this paper. I would like to thank the refree for useful comments.

My Ph.D. was funded by the DPMMS, the EPSRC and Selwyn College. I have been also supported by the Physics department of the University of Padova during two visits. Now, I am funded by the grant FIRB 2012 "Moduli Spaces and their applications" and by the ERC StG 307119 - StabAGDG. 
\end{subsection}
\end{section}

\begin{section}{Ind-varieties and commutative Monoids}\label{app}
In this section, we recall some general definitions and results about ind-varieties and monoid. A reference about ind-variety is \cite[Chapter IV]{Kumar}. An ind-variety $X$ is a collection $(X_n)_{n\geq 0}$ of algebraic varieties and a collection of closed embeddings
$$
\iota_n\colon X_{n-1}\hookrightarrow X_n
$$
We write
$$
X=\lim_n X_n
$$
This limit exists in the category of locally ringed spaces; however, we prefer to enlarge the category of schemes including all direct systems. This means that for us an ind-variety is a direct system of algebraic varieties.

A line bundle $L$ on $X$ is the data of a line bundle $L_n$ on each $X_n$ such that $\iota_n^*L_n=L_{n-1}$.  A section $s$ of $L$ is a collection of sections $(s_n)_{n\geq 0}$ such that $s_n$ is a section of $L_n$ on $X_n$ and the restriction  of $s_n$ to $X_{n-1}$ is $s_{n-1}$. We assume that the vector space $H^0(X,L^k)$ is finite dimensional for every $k$. The ring of sections of $L$ is thus defined as a projective limit in the category of graded rings
\[\RR(X,L):=\varprojlim \RR(X_n,L_n)\,.\]
We do not have to worry about the topology of this ring because of graded pieces $H^0(X,L^k)$ are finite dimensional. In other words, elements on $\RR(X,L)$ are not formal power series.

\begin{remark}[Ampleness on ind-varieties]\label{ample}
The concept of ampleness for a line bundle $L$ on an ind-variety $X$ is subtle and, to the best of our knowledge, there is not a standard definition. A first definition could be that there exists a $k$ such that $L^k_n$ is very ample on $X_n$ for every $n$. Remark that $k$ does not depend on $n$. If $H^0(X,L^k)$ is finite dimensional for every $k$ but the dimension of $X_n$ tends to infinity when $n$ grows, $L$ can not be ample. A weaker definition is to ask that for every $n$ there exists a $k=k(n)$ such that $L_n^k$ is very ample on $X_n$. The example that we will study in this paper is ample just in the sense of the second definition. This second definition does not imply the classical consequences of ampleness: for instance, in this generality, it is not even clear that an ample line bundle is effective.
\end{remark}

An ind-monoid is an ind-variety $M$ with an associative multiplication and an identity element $1_M$. A multiplication $m$ is a family of maps
\[m_{g,h}\colon M_g\times M_h \rightarrow M_{g+h}\]
compatible with the restrictions. $M$ is commutative if the multiplication is. 
\begin{definition} [Split monoid]\label{split}
Let $M$ be a commutative ind-monoid and $L$ a line bundle on $M$. We say that $M$ is split with respect to $L$ if the following two conditions hold:
\begin{enumerate}
\item For every $g$ and $h$
\[m_{g,h}^*L_{g+h}\cong pr_1^*L_g \otimes pr_2^*L_h=:L_g\boxtimes L_h\]
where $pr_i$ are the projections; 
\item for every $k$, the vector space $H^0(M,L^k)$ is finite dimensional and spanned by characters, where a section $\chi$ of $L$ is a character if
\[m_{g,h}^*\chi_{g+h}=\chi_g\boxtimes \chi_h  \qquad \forall \, g,h\,.\]
\end{enumerate}
\end{definition}
In the language of Hopf algebras, condition (1) means that the pull-back $m^*$ is a co-commutative co-multiplication for $\RR(M,L)$. The definition of character makes sense only if condition (1) holds. With a slight abuse of notations, we will speak about characters of $M$ rather than characters of the co-algebra $\RR(M,L)$, and we will write $\chi(\alpha\beta)=\chi(\alpha)\chi(\beta)$ instead of $m_{g,h}^*\chi_{g+h}(\alpha \times \beta)=\chi_g(\alpha)\boxtimes \chi_h(\beta) $.
\begin{lemma}\label{li}
Let $M$ be a commutative monoid, suppose it is split with respect to a line bundle $L$, then, any set of characters is linearly independent.
\end{lemma}
\begin{proof}
This proof is standard. We argue by contradiction. Take $n$ minimal such that there exist $n$ linearly dependent characters $\chi_1,\dots , \chi_n$. We can write
\[\chi_n=\sum_{i=1}^{n-1}\l_i\chi_i \qquad \l_i\in \C\,.\]
Pick $\alpha\in M$ such that $\chi_1(\alpha)\neq\chi_n(\alpha)$. For any $\beta\in M$ we have
\[\sum_{i=1}^{n-1}\l_i\chi_i(\alpha)\chi_i(\beta)=\chi_n(\alpha)\chi_n(\beta)=\chi_n(\alpha)\left(\sum_{i=1}^{n-1}\l_i\chi_i(\beta)\right)\]
Since $\beta$ is arbitrary we get
\[\sum_{i=i}^{n-1}\l_i(\chi_i(\alpha)-\chi_n(\alpha))\chi_i=0\,.\]
The coefficient $\chi_1(\alpha)-\chi_n(\alpha)$ is non-zero, so we have written a non-trivial linear relation among fewer than $n$ characters. This contradicts the minimality of $n$.
\end{proof}
\begin{proposition}\label{rep}
Let $M$ be a commutative ind-monoid and $N$ a submonoid. Suppose that $M$ is split with respect to a line bundle $L$. Then the ideal $I_N$ in $\RR(M,L)$ of sections vanishing on $N$ is generated by differences of characters
\[\chi_i-\chi_j\]
\end{proposition}
\begin{proof}
Take $s$ in $I_N$. We can assume that $s$ is homogeneous and write it as a linear combination 
\[s=\l_1\chi_1+\cdots +\l_n\chi_n\]
where $\chi_i$ are characters and $\l_i$ are constants. Restricting $\chi_i$ to $N$ some of them might become equal. Up to relabelling the $\chi_i$, we can fix integers $0=m_0 < m_1 < \cdots <m_k=n$ and distinct characters $\theta_1, \dots , \theta_k$ of $N$ such that
\[\chi_i\mid_N=\theta_j \quad \iff \quad m_{j-1}<i\leq m_j\]
For $j=1,\dots , k$, let us define 
\[\mu_j:=\sum_{i=m_{j-1}+1}^{m_j}\l_i \,.\]
By hypothesis we know that
\[0=s\mid_N \, =\sum_{j=1}^k \mu_j\theta_j\,.\]
By Lemma \ref{li} we have $\mu_j=0$ for every $j$, so
\[s=s-\sum_{j=1}^k\mu_j\chi_{m_j}=\sum_{j=1}^k \sum_{i=m_{j-1}+1}^{m_j}\l_i(\chi_i-\chi_{m_j})\]
The differences $\chi_i-\chi_{m_j}$ vanish on $N$ for $m_{j-1}<i\leq m_j$, so we have just expressed $s$ as linear combination of differences of characters vanishing on $N$.
\end{proof}
The previous argument actually shows that every element of the ideal can be written as a linear combination of differences of characters. These results are special cases of a more general theory of Milnor and Moore. They have many applications in the study of moduli spaces, e.g. \cite{Hul}.
\end{section}

\begin{section}{Satake compactification, modular forms and theta series}\label{satake}
We recall some facts about modular forms and the Satake compactification of $\A_g$. General references about modular forms are \cite{1-2-3} and \cite{Mum}. The Satake compactification was first defined in \cite{Satake}; a comprehensive reference is \cite{Fre}.

The line bundle $L_g$ of weight one modular forms on $\A_g$ is defined as the determinant of the Hodge bundle; it is ample and it generates the rational Picard group.
\begin{definition}[Siegel modular form]
A weight $k$ and degree $g$ Siegel modular form is a section of $L_g^k$ on $\A_g$.
\end{definition}
The universal cover of $\A_g$ is the Siegel upper half space $\H_g$; the symplectic group $Sp(2g,\mathbb{Z})$ acts on $\H_g$ and
\[\A_g=\H_g/Sp(2g,\mathbb{Z})\]
The line bundle $L_g$ is trivial when it s pulled back to $\H_g$; therefore a modular form can be also defined as a holomorphic function on $\H_g$ which transforms appropriately under the action of $Sp(2g,\mathbb{Z})$.

The Satake compactification $\A_g^S$ is a normal projective variety defined as follows
\[\A_g^S:=\Proj(\bigoplus_{n\geq 0} H^0(\A_g,L^n_g))\]
This is the compactification "seen" by modular forms. The line bundle $L_g$ extends naturally to $\A_g^S$ because it is the $\mathcal{O}(1)$ of this $\Proj$. For the same reason, all modular forms extend to $\A_g^S$

\begin{definition}[The Siegel operator]\label{Siegel}
The Siegel operator $\Phi$ is a map of graded rings
$$
\Phi \colon \bigoplus_{n \geq 0} H^0(\A_g, L^n_g)\rightarrow \bigoplus_{n\geq 0} H^0(\A_{g-1},L^n_{g-1})
$$
defined as 
\[\Phi(F)(\t):=\lim_{t\rightarrow+\infty}F(\t\oplus it)\,,\]
where $\t$ is an element of $\H_{g-1}$ and $t\in \mathbb{R}$. Here, we are thinking at $F$ as a holomorphic function on $\H_g$. 
\end{definition}
Clearly, there is some work to do to show that $\Phi(F)$ is a well defined element of $H^0(\A_{g-1},L_g^n)$; the interested reader cal look at \cite{Fre}.

The Siegel operator is surjective for $n$ even and larger than $2g$ (\cite{Fre} page 64); this means that the Siegel operator defines a closed embedding of $\iota_g\colon \A_{g-1}^S \hookrightarrow \A_g^S$. One can check that the image of $\A_{g-1}^S$ is the boundary $\partial \A_g^S$ of $\A_g^S$, so we obtain a stratification
\[\A_g^S=\A_g\sqcup \A_{g-1}^S= \A_g \sqcup \A_{g-1} \cdots \A_1 \sqcup \A_0\,.\]
By construction, the pull-back $\iota_g^*L_g$ is isomorphic to $L_{g-1}$, and the pull-back $\iota_g ^*\colon H^0(\A_g, L^n_g)\rightarrow H^0(\A_{g-1},L^n_{g-1})$ is the Siegel operator. Again, a reference is \cite{Fre}.

The system of varieties $\A_g^S$ together with the closed embeddings $\iota_g$ induced by the Siegel operators forms a direct system, so we can define the ind-variety
\[\A_{\infty}:=\lim_g \A_g^S\]
We follow the notations of Section \ref{app}. The line bundles $L_g$ define a line bundle 
$$L_{\infty}:=\lim_g L_g$$ 
on $\A_{\infty}$. This line bundle is called the line bundle of weight one stable stable modular forms. 
\begin{definition}[Stable modular forms]\label{StableModForm}
A weight $k$ stable modular form $F$ is a section of $L^k_{\infty}$. More concretely, it is a collection
\[F=(F_g)_{g\geq 0}\]
where $F_g$ is a modular form of weight $k$ on $\A_g$ and
\[\Phi(F_{g+1})=F_g\]
\end{definition}
Recall that each line bundle $L_g$ is ample on $\A_g$; however, the same assertion is problematic for $L_{\infty}$, as explained in Remark \ref{ample}.

We now define a structure of commutative monoid on $\A_{\infty}$. Given two principally polarised abelian varieties of dimension respectively $g$ and $h$, their product is still a principally polarised abelian variety of dimension $g+h$. This gives a commutative operation
\begin{displaymath}
 \begin{array}{cccc}
  m \colon&\A_{\infty}\times \A_{\infty}&\rightarrow &\A_{\infty}\\
   &([X],[Y])&\mapsto & [X\times Y]
 \end{array}
\end{displaymath}
The identity element is $\A_0$.
\begin{lemma}\label{pull_back_hodge}
 On $\A_{\infty}\times \A_{\infty}$, we have
\[m^*L_{\infty}=L_{\infty}\boxtimes L_{\infty}\,.\]
\end{lemma}
\begin{proof}
For every pair of integers $g$ and $h$ one looks at the morphism
$$m\colon \A_g\times \A_h\to \A_{g+h}  $$
The fibre of the Hodge bundle $E_g$ at a point $[X]$ of $\A_g$ is the tangent space at the identity of $X$. This implies that $m^*E_{g+h}$ is $E_g\boxtimes E_h$. The statement on $L_g$ follows by taking the determinant. 
\end{proof}
 We have the following useful formal consequence
\begin{proposition}\label{hopf}
The pull-back $m^*$ defines a co-commutative co-multiplication on the algebra of stable modular forms $\RR(\A_{\infty},L_{\infty})$.
\end{proposition}

The ring of stable modular form, so far, could be trivial. However, there is a classical and surprising way to produce plenty of stable modular forms out of quadratic forms. Let us go trough all definitions. A quadratic form is a pair $(\L,Q)$, where $\L$ is a finitely generated free group and $Q$ is a $\mathbb{Z}$-valued bilinear form on $\L$. The rank of the quadratic form is defined as the rank of $\L$. The elements of $\L$ are called vectors, and the norm of a vector $v$ is $Q(v,v)$. We always assume $Q$ to be even (i.e. $Q(v,v)$ is even for every $v$), unimodular (i.e. $\det Q=1$) and positive definite. Often, we will denote a quadratic form just by $Q$.
\begin{definition}[Theta series]\label{thetaseries}
Let $(\L,Q)$ be an even unimodular positive definite quadratic form and $g$ a positive integer, the associated theta series is
\[\Theta_{Q,g}(\t):=\sum_{x_1,\dots ,x_g\in \L}\exp(\pi i \sum_{i,j}Q(x_i,x_j)\t_{ij})\]
where $\t$ belongs to $\H_g$.
\end{definition}
This is a weight $\frac{1}{2}\rk(\L)$ and degree $g$ modular form. By explicit computation, one sees that the Siegel operator \ref{Siegel} acts as follows
\[\Phi(\Theta_{Q,g+1})=\Theta_{Q,g}\,,\]
so the collection of all theta series
\[\Theta_{Q}:=  (\Theta_{Q,g})_{g\geq 0}\]
is a stable modular form. Given $X\in \A_g$ and $Y\in \A_h$, we have the factorisation property
\[\Theta_{Q,g+h}([X\times Y])=\Theta_{Q,g}(X)\Theta_{Q,h}(Y)\,,\]
which means that the \emph{theta series are characters} for the monoid $\A_{\infty}$. 
\begin{example}[Quadratic forms and theta constants]\label{Wittquadratic forms}
In some cases, theta series can be written in term of theta constants, let us give some examples following \cite{IgusaSchot}. Let $E_8$ be the quadratic form associate to the Dynkin diagram $E_8$. Using a similar definition, for every integer $k$ one can define the Witt quadratic forms $W_{8k}$. The quadratic form $W_{8k}$ has rank $8k$, it is equal to $E_8$ for $k=1$ and to $D_{16}^+$ for $k=2$. Up to a constant, we have the following expansion
\[\Theta_{W_{8k},g}(\t)=\sum_{\epsilon \,\textrm{even}}\theta[\epsilon]^{8k}(\t)\,,\]
where the sum runs over all the even theta characteristics. In particular, the well-known Schottky form can be written as
\[\Theta_{D_{16}^+}-\Theta_{E_8\oplus E_8}=2^{-g}\sum_{\epsilon \,\textrm{even}}\theta[\epsilon]^{16}(\t)-2^{-2g}(\sum_{\epsilon \,\textrm{even}}\theta[\epsilon]^{8}(\t))^2\]
In general, a theta series will not have such a simple expression in term of theta constants. In \cite[Section 3]{StableSalvatiManni}, there is a systematic analysis of the theta series which can be expanded in this way; the results of that paper relies upon \cite[Theorem 6.3]{Mum}.
\end{example}

The ring of stable modular forms is described by the following result of Freitag:
\begin{theorem}[Theorem 2.5 of \cite{stabile}]
The ring of stable modular forms $\RR(\A_{\infty},L_{\infty})$ is the polynomial ring in the theta series associated to irreducible quadratic forms.
\end{theorem}
Freitag's main contribution was to show that  $H^0(\A_{\infty},L_{\infty}^k)$ is spanned by theta series for every $k$. There are finitely many quadratic forms of a given rank, so we already learn that $H^0(\A_{\infty},L_{\infty}^k)$ is finite dimensional. This result, together with Lemma \ref{pull_back_hodge} and the fact that theta series are characters, means that $\A_{\infty}$ equipped with the line bundle $L_{\infty}$ is a split monoid in the sense of Definition \ref{split}. Freitag's claim about the polynomial structure now follows easily from Proposition \ref{rep}.
\end{section}

\begin{section}{Satake compactification of the hyperelliptic locus}\label{Hyperelliptic}
In this section we define the Satake compactification of the hyperelliptic locus and we prove Theorems \ref{Main_Intro} and \ref{Transversality_Intro}. Consider the Jacobian morphism
$$ j\colon \h_g \to \A_g   $$
mapping a curve to its Jacobian.
\begin{definition}[Satake compactification]
The Satake compactification $\h_g^S$ of the hyperelliptic locus $\h_g$ is the scheme-theoretic closure of $j(\h_g)$ inside $\A_g^S$.
\end{definition}
A degeneration of a hyperelliptic Jacobian is still the Jacobian of a hyperelliptic curve (\cite{Hoyt}), so we have a stratification
$$ \h_g^S=\bigsqcup_{\sum g_i\leq g}\h_{g_1}\times\cdots \times \h_{g_k}$$
Equivalently, the Satake compactification is the image of the Deligne-Mumford compactification of $\h_g$ under the morphism which maps a curve to the Jacobian of its normalisation.

\smallskip

In particular, $\h_{g+1}^S$ contains $\h_g^S$ as a scheme, so we can define the commutative ind-monoid 
 \[\h_{\infty}:=\lim_g \h_g^S\]
Using the monoid structure we can show the following
\begin{theorem}\label{differences}
The ideal of stable modular forms vanishing on $\h_{\infty}$ is generated by differences of theta series
$$\Theta_P-\Theta_Q$$
where $P$ and $Q$ are even, unimodular, positive definite quadratic forms of the same rank.
\end{theorem}
\begin{proof} 
We know that $\h_{\infty}$ is a commutative sub-monoid of $\A_{\infty}$ and $\A_{\infty}$ satisfies the hypotheses of Definition \ref{split}. The theta series are the characters of co-algebra $\RR(\A_{\infty},L_{\infty})$, so the result is a direct consequence of Proposition \ref{rep}.
\end{proof}
So far, the ideal studied in the theorem could be trivial. A key tool to show that a modular form is a stable equation for the hyperelliptic locus is the following geometric result. Let $\A_g^{ind}$ be the moduli space of indecomposable principally polarised abelian $g$-fold.

\begin{theorem}[Transversality]\label{Transversality}
Inside $\A_{g+1}^S$ , the intersection of $\h_{g+1}^S$ and $\A_g^{ind}$ is scheme theoretically equal to $\h_g$.
\end{theorem}
The statement was well-known at the level of sets (cf \cite{Hoyt} or \cite[Lemma 11.6.14]{ACG}); we call it a transversality result because it states that the scheme structure of the intersection is the reduced one. 
\begin{proof}
We work with level structure $(4,8)$, and we denote by $\A_g^S(4,8)$ the Satake compactification of the moduli space $\A_g(4,8)$ of principally polarised abelian $g$-fold with level structure $(4,8)$. This amounts to take a finite Galois cover of $\A_g^S$. Now, $\A_{g+1}^S(4,8)$ has several boundary components, all isomorphic to $\A_g^S(4,8)$. We will fix one of them, let us call it $V$. The hyperelliptic locus $\h_{g+1}^S(4,8)$ breaks into several irreducible components (cf. \cite{Tsu}); an irreducible component is identified by the choice of a fundamental system $\m$ of theta characteristic, we will fix such an $\m$ and denote by $Y=Y_{\m}$ the corresponding irreducible component. Since the cover is Galois, locally the intersection of $\h_{g+1}^S$ and $\A_g$ is isomorphic to the intersection of $Y$ and $V$. Because of this, it is enough to show that the scheme-theoretic intersection of $V^{ind}$ and $Y$ is reduced. We need to work at level $(4,8)$ to apply the following result, which is due set-theoretically to Mumford and scheme theoretically to Salvati Manni.
\begin{theorem}[\cite{SM}]\label{SM}
Fix a fundamental system of theta characteristic  $\m=(m_1, \dots , m_{2g+1})$, and let $b$ be the sum of  the odd $m_i$. Then, the corresponding irreducible component $Y=Y_{\m}$ of $\h_g(4,8)$ is scheme theoretically defined by the vanishing of the theta constants $\theta_{m+b}$ such that $m = m_{i_1} + \cdots + m_{i_k}$  for $k<g$ and $k \equiv g,g+1$, and the non-vanishing of the remaining theta constants.
\end{theorem}
In the statement, the non-vanishing of the remaining theta constants is needed to rule out the loci of decomposable abelian varieties; our statement is about Jacobians of smooth curves, so we are already working outside these loci. Let $\mathfrak{m}$ be a $g+1$ dimensional system of fundamental theta characteristic, e.g. the one defined in equation $7$ of \cite{SM}. Let $Y=Y_{\mathfrak{m}}$ be the corresponding irreducible component of $\h_{g+1}^S(4,8)$. Let $I_{g+1}(Y_{\mathfrak{m}})$ be the ideal of modular forms generated by the theta constants vanishing along $Y_{\mathfrak{m}}$. Because of Salvati Manni's result, this ideal defines scheme-theoretically $Y_{\mathfrak{m}}$. By direct computation one sees that
\[\Phi(\theta
\left[
\begin{array}{cc}
\epsilon & 0\\
\epsilon' & \delta
 \end{array}
\right])=
\theta
\left[
\begin{array}{c}
\epsilon \\
\epsilon'
 \end{array}
\right]\]
for $\delta$ equal either to 0 or 1. In the previous formula, $\Phi$ is the Siegel operator, i.e. the restriction operator from $\A^S_{g+1}(4,8)$ to one of the boundary component, say to $V$. Because of the theorem quoted above, this means that, scheme theoretically, the intersection of $Y_{\mathfrak{m}}$ and $V\cong \A_g(4,8)^S$ away from decomposable abelian varieties is isomorphic to $Y_{\mathfrak{n}}$; where $\mathfrak{n}$ is the $g$ dimensional system of fundamental theta characteristic defined in equation $7$ of \cite{SM}.
\end{proof}

We now complete the description of the tangent space of $\h_{g+1}^S$ along $\h_g$. Let $C$ be a smooth genus $g$ hyperelliptic curve and $X$ its Jacobian. To start with, let us describe the normal bundle exact sequence of $\A_g$ in $\A_{g+1}^S$ at $X$. This sequence reads
$$
 0\to T_X\A_g\to T_X\A_{g+1}^S\to H^0(X,2\Theta)^{\vee}\to 0  
$$
We will need the following explicit description of the action of these derivations. Let $F_{g+1}$ be a modular form on $\A_{g+1}$ and $F_g$ its restriction to $\A_g$. For any element $T$ in the Siegel upper half space $\H_{g+1}$ write
$$
T=
\left(
\begin{array}{cc}
\t & z \\ 
 ^t z & t
\end{array}
\right)\,,
$$
with $t$ in $\H_1$ and $\t$ in $\H_g$. Let $q:=\exp (2\pi i t)$; then the \emph{Fourier-Jacobi} expansion of $F_{g+1}$ is
\begin{equation}\label{FJ}
F_{g+1}(T)=F_g(\tau)+\sum_{n\geq 1}f_n(\tau,z)q^n
\end{equation}
where $f_n$ is a section of $H^0(X_{\tau},2n\Theta)$, $X_{\tau}$ is the principally polarised abelian variety defined by $\tau$, and $z$ is a system of co-ordinates on $X_{\tau}$. A derivation $D\in T_X \A_g$ acts as $D(F_{g+1})=D(F_g)$; a derivation in $D\in H^0(X,2\Theta)^{\vee}$ acts as $D(F_{g+1})=D(f_1)$.

The normal bundle exact sequence for $\h_g$ in $\h_{g+1}^S$ at $C$ is a subsequence of the normal bundle exact sequence of $\A_g$ in $\A_{g+1}^S$ at $X$. To describe it, we need to introduce the following morphism
\begin{equation}\label{Psi}
\begin{array}{cccccc}
\Psi:& C & \xrightarrow{f} & C\times C &\xrightarrow{\delta} & X\\ 
    & p & \mapsto & (p,\iota(p))&&\\
&&&(a,b)&\mapsto &AJ(a)-AJ(b)
\end{array}
\end{equation}
where $\iota$ is the hyperelliptic involution and $AJ$ is the Abel-Jacobi map. 
\begin{lemma}
The pull-back $\Psi^*2\Theta$ is isomorphic to $2(K_C+W)$, where $W$ is the divisor of Weierstrass points on $C$.
\end{lemma}
\begin{proof}
The pull-back $\delta^*2\Theta$ is $K_C\boxtimes K_C(2\Delta)$, where $\Delta$ is the diagonal; this is well-known, e.g. \cite[Equation 4.4]{Welters}. Now, we pull-back $K_C\boxtimes K_C(2\Delta)$ via $f$. The pull-back of $\Delta$ is $W$; the pull-back of $K_C\boxtimes K_C$ is $K_C+\iota^*K_C=2K_C$. 

\end{proof}

\begin{theorem} \label{tsHyp}
Keep notation as above, the normal bundle exact sequence for $\h_g$ in $\h_{g+1}$ at $C$ is
 \[0\rightarrow T_C\h_g \rightarrow T_C\h_{g+1}^S\rightarrow P_C \rightarrow 0\]
where $P_C$ is the image of the map
\[\Psi^*\colon H^0(X,2\Theta)\rightarrow H^0(C,2(K_C+W))\,.\]
In other words, the normal tangent cone at $C$ is the cone over $(\Psi(C),2\Theta)$.
\end{theorem}
\begin{proof}
There are two things we have to prove, first
\[T_C\h_{g+1}^S\cap T_C\A_g=T_C\h_g\,;\]
but this is equivalent to Theorem \ref{Transversality}.

To describe the co-kernel of the inclusion
$$T_C\h_g\hookrightarrow T_C\h_{g+1}^S$$
we need to know that, after blowing up $\A_g^S$ in $\A_{g+1}^S$, the proper transform of $\h_{g+1}^S$ meets the Kummer variety of $X$ in $\Psi(C)$. This is proved in \cite[Theorem 6]{Namik}, just remark that to obtain a generic irreducible nodal hyperelliptic curve we need to glue two points conjugated under the hyperelliptic involution.
\end{proof}
In \cite[Lemma 3.6]{PhD}, it is shown that $\Psi^*$ is not surjective and $P_C$ has rank $2g$; however, we do not need this result here.

\begin{remark}[Failure of Theorem \ref{Transversality} at the locus of decomposable abelian varieties]
For the sake of completeness, let us sketch a proof of the failure of Theorem \ref{Transversality} at the locus of decomposable abelian varieties. This result is not needed in this paper, but we think that the study of this intersection is interesting on its own. Pick two integers such that $g_1+g_2=g$; fix a hyperelliptic curve $C$ of genus $g_1$ and a hyperelliptic curve $D$ of genus $g_2$. Call $\iota$ the hyperelliptic involution. The point $(C,D)$ in $\h_{g+1}^S$ represents all the hyperelliptic curves of the form $C\sqcup D/(p\sim q,\iota(p)\sim \iota(q))$, where $p$ is a point varying in $C$ and $q$ is varying in $D$. Recall that we have an identification
$$\Sym ^2 (H^0(C,K_C)^{\vee}\oplus H^0(D,K_D)^{\vee}) =T_{(J(C)\times J(D))}\A_g $$
Arguing as in \cite{CSB}, we can show that the tangent space of $\h_{g+1}^S$ at $(C,D)$ contains the image of the map
\begin{displaymath}
\begin{array}{cccc}
 \psi\colon& C/\iota\times D/\iota& \to &\P \Sym ^2 (H^0(C,K_C)^{\vee}\oplus H^0(D,K_D)^{\vee}) \\
& (p,q)& \mapsto& \o_i(p)\psi_j(q)+\o_j(p)\psi_i(q) 
\end{array}
\end{displaymath}
where $\o_i$ are a basis of $H^0(C,K_C)$ and $\psi_i$ are a basis of $H^0(D,K_D)$. This is the same tangent direction we get when we consider an appropriate smoothing of the genus $g$ nodal curve $C\sqcup D/p\sim q$; this smoothing is not hyperelliptic, so the intersection of $T_{(C,D)}\h_{g+1}^S$ with  $T_{(J(C)\times J(D))}\A_g $ is strictly bigger than $T_{(C,D)}\h_g^S$.
\end{remark}
\end{section}

\begin{section}{Projective invariants of hyperelliptic curves}
In this section, we review the theory of projective invariants, and we use it to show that certain modular forms vanish on the hyperelliptic locus. 

To start with, let us introduce the auxiliary space $\B_g$. This is the moduli space of $2g+2$ points on $\P^1$, up to permutation and projectivity. The points are counted with multiplicity, points are not allowed to have multiplicity bigger that $g+1$. This space is classically constructed as a GIT quotient; it is irreducible, it has an open dense subset $\B_g^{\circ}$ where the $2g+2$ points are all distinct, and a boundary $D$ where at least two points coincide.

Let $C$ be a smooth genus $g$ hyperelliptic curve, fix a two to one map $\pi\colon C \to \P^1$. This morphism is unique up to projective transformations of $\P^1$, it ramifies at $2g+2$ points. A point $p$ is called a Weierstrass point if it is a ramification point for $\pi$.
\begin{definition}[Projective invariants]
The projective invariants of $C$ are the image of the Weierstrass points under $\pi$, considered up to permutations and projective automorphisms of $\P^1$.
\end{definition}
Equivalently, the projective invariants of $C$ are the points of the branch divisor of $\pi$, considered up to projectivity. The projective invariants of a smooth hyperelliptic curve $C$ are naturally a point of $\B_g^{\circ}$, so we have a morphism 
$$f_g\colon \h_g \to \B_g^{\circ}$$
One can reconstruct a hyperelliptic curve out of its projective invariants, and given $2g+2$ points on $\P^1$ there is a hyperelliptic curve with that projective invariants; this means that $f_g$ is an isomorphism. 

This construction has been extensively used to study the moduli space $\h_g$; references are \cite{IPI}, \cite{AL02} and \cite[Chapter 2]{Pas}.

As an aside, let us recall that the Thom\ae's formula permits to write the cross-ratios of the projective invariants in term of second order theta functions evaluated at the period matrix of $C$.
 
Following \cite{AL02}, the map $f_g$ extends to an isomorphism
$$f_g\colon \h_g \sqcup \eta_0^*\to  \B_g^{\circ} \sqcup D^*$$
where $\eta_0^*$ parametrises irreducible singular hyperelliptic curves with just one node, i.e. curves of the form $C/p\sim \iota(p)$, and $D^*$ parametrises $2g+2$ points on $\P^1$ such that exactly $2$ points coincide. The image of a curve in $\eta_0^*$ is a set of $2g+2$ points of the form $\{p_1,\dots , p_{2g},p,p\}$, where $\{p_1,\dots , p_{2g}\}$ are the projective invariants of the normalisation and the glued points are the pre-images of $p$ under $\pi$. 

Always following \cite{AL02}, we can extend $f_g$ further to a morphism
$$f_g\colon \overline{\h}_g\to \B_g$$
which contracts the boundary divisors of the Deligne-Mumford compactification different from the closure of $\eta_0^*$ to high co-dimension loci.  

\smallskip

We need some more information about $\B_g$, again references are \cite{IPI}, \cite{AL02} and \cite[Chapter 2]{Pas}. From the GIT point of view, the moduli space $\B_g$ can be constructed as the $\Proj$ of the ring $S(2,2g+2)$, which is the ring of co-invariant of binary forms of degree $2g+2$. This ring is formally constructed as follows: let $Z$ be the cartesian product of $2g+2$ copies of $\P^1$, on this variety we have a diagonal action of $SL(2,\C)$ and an action of the symmetric group; this action linearise to an action on the line bundle $M:=\mathcal{O}(1,\dots, 1)$, the ring $S(2,2g+2)$ is the ring of invariant element of $\RR(Z,M)$. More concretely, $S(2,2g+2)$ is the ring of symmetric functions in $2g+2$ variables, which are co-invariant under the natural action of $SL(2,\C)$. The discriminant $\D$ is an element of $S(2,2g+2)$ of degree $4g+2$, it cuts out the boundary divisor $D$. 

\smallskip

We now consider also the Jacobian morphism
$$ j\colon \overline{\h}_g\to \h_g^S \hookrightarrow  \A_g$$
Take the composition
$$j\circ f_g^{-1}\colon \B_g \dashrightarrow   \A_g $$ 
Under this map, $D^*$ dominates $\h_{g-1}$; explicitly, the set $\{p_1,\dots , p_{2g},p,p\}$ is mapped to the Jacobian of the smooth hyperelliptic curve defined by the projective invariants $\{p_1,\dots , p_{2g}\}$. Taking the pull-back we obtain a morphism of graded ring
\[\rho\colon\RR (\A_g,L_g)\rightarrow S(2,2g+2)\]
whose kernel is exactly the ideal of modular forms vanishing on the hyperelliptic locus. This map is sometime called Igusa morphism of projective invariants, it was introduced in \cite{IPI}, where Igusa proved that its degree is $\frac{1}{2}g$ \footnote{If $g$ is odd, all non-trivial modular forms have odd degree, so the factor $\frac{1}{2}$ should not worry the reader.}. Using this construction we can prove the following criterion.
\begin{criterion}[Weissauer - unpublished]\label{crit_slope}
Let $F_g$ be a weight $n$ and degree $g$ modular form. Restrict it to $\h_g^S$ and say it vanishes along $\h_{g-1}$ with multiplicity at least $k$. If
\[\frac{n}{k}< 8+\frac{4}{g}\]
then $F_g$ vanishes on $\h_g$.
\end{criterion}
\begin{proof} 
Suppose $F_g$ vanishes with multiplicity at least $k$ on $\h_{g-1}$. This means that 
$(j\circ f_g^{-1})^*F_g$ vanishes with multiplicity at least $k$ on $D$. In other words, $\D^k$ divides $\rho(F_g)$.
The degree of the discriminant in $S(2,2g+2)$ is $4g+2$, the degree of $\rho(F_g)$ is $\frac{1}{2}gn$.
Since, by hypothesis,
\[k(4g+2)>\frac{1}{2}gn\]
we obtain that $\rho(F_g)$ is equal to zero, so the claim.
\end{proof}
\begin{remark}[Relation with Theorem \ref{tsHyp}]
To show that $F_g$ vanishes along $\h_{g-1}$ with multiplicity at least $2$ one needs to know the tangent space of $\h_g^S$ along $\h_{g-1}$; in the applications, especially in Theorem \ref{second_equations}, we will use the description given in Theorem \ref{tsHyp}.
\end{remark}
\begin{remark}[Other versions of Criterion \ref{crit_slope}]
A weaker version of Criterion \ref{crit_slope} can be found in \cite{Poor}. An alternative proof is in \cite{Pas}. In \cite{SM2}, Salvati Manni attributed this criterion to Weissauer, in an unpublished preprint, and showed that the inequality is sharp. This Criterion is also related to the slope of the hyperelliptic locus, in the sense of slope of the cone of effective divisors (cf. \cite{CH}).
\end{remark}

\smallskip

Combining Theorem \ref{Transversality_Intro} and Criterion \ref{crit_slope} we can find a first group of stable equations for the hyperelliptic locus. We will need the following basic invariant of a quadratic form $(Q,\L)$
\[\mu(Q):=\min\{Q(v,v) \mid v\in \L ; v\neq 0\}=\min\{n \mid \RR_n(\L)\neq \varnothing \}\,,\]
where $\RR_n(\L)$ is the set of vectors of $\L$ of norm $2n$.
\begin{theorem}\label{eq_hyp}
Let $(Q,\L)$ and $(P,\G)$ be two even positive definite unimodular quadratic forms of rank $N$ and let $\mu:=\min \{\mu(Q),\mu(P)\}$. If
\[\frac{N}{\mu}\leq 8\,,\]
then 
\[F:=\Theta_{Q}-\Theta_{P}\]
is a stable equation for the hyperelliptic locus. In other words, $F_g$ vanishes on $\h_g$ for every $g$.
\end{theorem}
\begin{proof}
The proof is by induction on $g$. The difference of two theta series vanishes on $\A_0$. Suppose the statement true for $g$, we want to apply Criterion \ref{crit_slope} to $F_{g+1}$. 
Call $k:=\frac{1}{2}\mu$, we need to prove that $F_{g+1}$ vanishes at the boundary component $\h_g$ with multiplicity at least $k$.

We first compute the multiplicity along tangent direction parallel to the boundary, namely along $T_C\h_{g+1}^S\cap T_C\A_g$, where $C$ is a smooth genus $g$ hyperelliptic curve. This intersection is, by Theorem \ref{Transversality_Intro}, equal to $T_C\h_g$. By induction, $F_g$ vanishes along $\h_g$, so $F_{g+1}$ is annihilated by the derivations contained in $T_C\h_{g+1}^S\cap T_C\A_g$.

Let us now look at the normal direction to $\A_g$; we will use the Fourier-Jacobi expansion introduced in the equation (\ref{FJ}). Writing out the Fourier-Jacobi expansion of $F_{g+1}$, the hypothesis on $\mu$ implies that the first $k$ terms vanish. This means that $F_{g+1}$ vanishes with order at least $k$ along the normal direction to $\A_g$ in $\A_{g+1}^S$; in particular, we obtain that it vanishes along $\h_g$ with multiplicity at least $k$ and we can apply Criterion \ref{crit_slope}.
\end{proof}

The hypotheses of Theorem \ref{eq_hyp} are quite restrictive; let us describe the cases where the Theorem can be applied.
\begin{proposition}
Let $(Q,\L)$ be an even, positive definite, unimodular quadratic form, then
\[\frac{\rk(Q)}{\mu(Q)}\leq 8\]
if and only if the pair $(\rk(Q),\mu(Q))$ is equal to one of the following pairs: $(8,2)$, $(16,2)$, $(32,4)$ or $(48,6)$.
\end{proposition}
\begin{proof}
Given any even unimodular quadratic form $(Q,\L)$, there is an upper bound
\[\mu(Q)\leq 2\lfloor\frac{\rk(Q)}{24}\rfloor+2\]
where ``$\lfloor \, \rfloor $'' is the round down (see \cite[Section 7.7 Corollary 21]{SpherePack} ). This bound, combined with the fact that the rank is divisible by $8$, gives the Proposition.

\end{proof}
 
Let us call the type of  quadratic form the pair $(\rk(Q),\mu(Q))$. There is just one quadratic form of type $(8,2)$ and one of type $(24,4)$, so we do not get any stable equation in these cases. The rank $16$ case was considered by Poor in \cite{Poor}: there are two quadratic forms of type $(16,2)$, so one gets one equation. In \cite[Corollary 5]{King}, using a generalization of the mass formula, it is shown that there exist at least ten millions of quadratic forms of rank $32$ and $\mu=4$ ; however, just $15$ of them are known explicitly. The situation for quadratic forms of type $(48,6)$ is not clear: believably, there exist many of them, see \cite[Page 15]{King}, but there is not any lower bound and just $3$ of them are known explicitly. (King adopts a slightly different notation: every quadratic form is tacitly assumed to be positive definite.) To summarise, Theorem \ref{eq_hyp} can be applied to the following cases

\begin{corollary}\label{first_equations}
If $P$ and $Q$ are two even, unimodular, positive definite quadratic forms meeting one of the following two hypotheses:
 \begin{enumerate}
\item $\rk(P)=\rk(Q)=32$ and $\mu(P)=\mu(Q)=4$; that is, the quadratic forms do not have any vector of norm $2$;
\item $\rk(P)=\rk(Q)=48$ and $\mu(P)=\mu(Q)=6$; that is, the quadratic forms do not have any vector of norm $2$ and $4$;
\end{enumerate}
then, the difference
$$\Theta_P-\Theta_Q$$
is a stable equation for the hyperelliptic locus.
\end{corollary}
\end{section}

\begin{section}{Niemeier quadratic forms}\label{24}
Niemeier quadratic forms are rank $24$ quadratic forms. In this section we prove the following:
\begin{theorem}\label{second_equations}
Let $(P,\Gamma)$ and $(Q,\Lambda)$ be two rank $24$ quadratic forms with the same number of vectors of norm $2$, then the difference
$$\Theta_P-\Theta_Q$$
is a stable equation for the hyperelliptic locus.
\end{theorem}
Vectors of norm $2$ are usually called roots. sThis result concerns the following $5$ pairs of quadratic forms
\vspace{0.2 cm}
\begin{center}
\begin{tabular}{c|c|c|c|c|c}
 quadratic forms & $A_5^4D_4$ , $D_4^6$ &$A_9^2D_6$ , $D_6^4$ & $A_{11}D_7E_6$, $E_6^4$& $A_{17}E_7$ , $D_{10} E_7^2$ & $E_8D_{16}$, $E_8^3$ \\ 
\hline
\# roots & 72 & 120 & 144 & 216 & 360
\end{tabular}
\end{center}
\vspace{0.2 cm}
\indent where a quadratic form of rank 24 is labelled by its root system (see e.g. \cite[Section 3]{LatCod} for more details). The pair $E_8^ 3,D_{16}E_8$ corresponds to the modular form $\Theta_{E_8}(\Theta_{E_8 \oplus E_8}-\Theta_{D_{16}^+})$, so its behaviour was well-understood. The others cases can not be expressed as a product of lower weight stable modular forms and they are not covered by previous results. 

This result is surprising because the slope of these quadratic forms, i.e. the ratio between the rank and the norm of the shortest vector, is strictly bigger than $8$, so they were not expected to vanish on the hyperelliptic locus in every genus. Before proving our theorem we need two preliminary results.

\begin{subsection}{A formula for sections of $2\Theta$} Let $s$ be a section of $2\Theta$ on the Jacobian of a curve $C$ with period matrix $\t$. For every couple of points $a$ and $b$ of $C$ the following classical formula holds:
\begin{equation}\label{cf}
s(\t,a-b)=E(a,b)^2[s(\t,0)\o(a,b)+\sum_{i,j}\frac{\partial^2 s}{\partial z_i \partial z_j}(\t,0)\o_i(a)\o_j(b)]
\end{equation}
where $E$ is the Prime form, $\{\o_i\}$ is the basis of the holomorphic differentials on $C$ corresponding to the basis $\{\frac{\partial}{\partial z_i}\}$ of the tangent space at the origin of the Jacobian, $\o(a,b)$ is the fundamental normalised bi-differential, and everything is trivialised with respect to a choice of local co-ordinates $z_a$ and $z_b$. This formula is well known, see e.g. \cite[Appendix A]{MVdeg}.


\end{subsection}
\begin{subsection}{The ``heat equation'' for Niemeier quadratic forms} The classification of rank $24$ quadratic forms is due to Niemeier, but it has been simplified by Venkov proving and using the following identity
\begin{theorem}[Venkov, cf. \cite{LatCod} Section 3]  \label{polLat}
Let $(\L,Q)$ be a rank $24$ quadratic form, then
$$r_2(\L)Q(v,w)=8\sum_{y\in \RR_2(\L)}Q(y,w)Q(y,v) \qquad \forall \, v,w \in \L$$
where $r_2(\L)$ is the number of roots and $\RR_2(\L)$ is the set of roots.
\end{theorem}
The proof relies upon the theory of degree 1 modular forms with harmonic coefficients. Let us draw a consequence of this result about the Fourier-Jacobi expansion of theta series (cf. equation (\ref{FJ})).

\begin{corollary} [Heat equation]\label{he}
The first Fourier-Jacobi coefficient $f_1$ of a theta series associated to a rank $24$ quadratic form $\L$ satisfies the following ``heat equation''
 \[r_2(\L)\pi i\frac{\partial f_1}{\partial \t_{ij}}(\t,0)=3(1+\delta_{ij})\frac{\partial^2 f_1}{\partial z_i \partial z_j}(\t,0)\,,\]
where $r_2(\L)$ is the number of roots of $\L$.
\end{corollary}
\begin{proof}
By explicit computation, we can write out the first Fourier-Jacobi coefficient of a theta series:
\[f_1(\t,z)=\sum_{x_1,\dots , x_g \in \L}\sum_{y\in \RR_{2}(\L)}\exp( \pi i \sum_{i,j}Q(x_i,x_j)\t_{ij}+2\pi i \sum_iQ(y,x_i)z_i)\,.\]
Fix two indexes $i$ and $j$, by explicit computation we have
\[\frac{\partial^2f_1}{\partial z_i \partial z_j}(\t,0)=(2\pi i)^2 \sum_{x_1,\dots,x_g\in \L}\sum_{y\in \RR_2(\L)}Q(y,x_i)Q(y,x_j)\exp(\pi i \sum_{i,j}Q(x_i,x_j)\t_{ij})\,,\]
On the other hand
\[(1+\delta_{ij})\frac{\partial f_1}{\partial \t_{ij}}(\t,0)=2\pi i\sum_{x_1,\dots,x_g\in \L}Q(x_i,x_j)\exp( \pi i\sum_{i,j}Q(x_i,x_j)\t_{ij})\,,\]
the coefficient $(1+\delta_{ij})$ is because the variables on $\H_g$ are $\t_{ij}$ with $i\leq j$, so when we compute the derivative with respect to $\t_{ij}$ we need to derive both $\t_{ij}$ and $\t_{ji}$. Applying Theorem \ref{polLat} we obtain the result.
\end{proof}
This formula is also discussed in \cite[page 16]{MVdeg}. This result is generalised to higher order Fourier-Jacobi coefficients and higher rank quadratic forms in \cite[Theorem 10.3]{PhD}. 
\end{subsection}

\begin{subsection}{Proof of Theorem \ref{second_equations}} We want to show that $F_g=\Theta_{P,g}-\Theta_{Q,g}$ is zero on $\h_g$ for every $g$; we argue by induction on $g$. The case $g=0$ is easy. To prove the inductive step we use Criterion \ref{crit_slope}: we need to show that $F_{g+1}$ vanishes along $\h_g$ with multiplicity at least $2$. As in Theorem \ref{eq_hyp}, the derivative along directions tangent to $\A_g$ vanishes because of Theorem \ref{Transversality_Intro} and the inductive hypothesis.

The normal direction is quite different: now the first Fourier-Jacobi coefficient is not trivial, so there is some work to do. Because of Theorem \ref{tsHyp}, it is enough to check that $f_1$ vanishes when restricted to points of the form $(\t,p-\iota(p))$, where $\t$ is the period matrix of a smooth hyperelliptic curve $C$, $p$ is a point of $C$ and $\iota$ is the hyperelliptic involution. 

To show this we argue as follows. First remark that
\[f_1(\t,0)=F_g(\t)=0\]
Then we apply the formula (\ref{cf}), trivialising everything with respect to co-ordinates $z_p$ and $\iota^*z_p$ and recalling that
\[\frac{\o}{dz_p}(p)=\frac{\o}{\iota^*dz_p}(\iota(p))\]
we get
\[f_1(\t,p-\iota(p))=E(p,\iota(p))^2\sum_{i,j}\frac{\partial^2 f_1}{\partial z_i \partial z_j}(\t,0)\o_i(p)\o_j(p)\]
Now the heat equation \ref{he} and the hypothesis on the number of roots come into the game: since $r_2(\L)=r_2(\G)=:r$, we have 
\[6\sum_{i,j}\frac{\partial^2 f_1}{\partial z_i \partial z_j}(\t,0)\o_i(p)\o_j(p)=r\pi i \sum_{i\geq j}\frac{\partial F_g}{\partial \t_{ij}}(\t)\o_i(p)\o_j(p)=(r\pi i) dF_g(\t)(p)\]
 Let us explain the last equality: the fibre of the cotangent bundle of $\A_g$ at $Jac(C)$ is isomorphic to $\Sym^2H^0(C,K_C)$, so $dF_g(\t)$ is a quadric in $\P H^0(C,K_C)^{\vee}$ and we can evaluate it on the image of $p$ under the canonical map. 

The co-normal bundle of $\h_g$ in $\M_g$ is given by the $-1$ eigenspace of $H^0(C,2K_C)$; the image of the co-differential $m\colon \Sym^2H^0(C,K_C)\to H^0(C,2K_C)$ is the $+1$ eigenspace; we conclude that the quadric in the co-normal bundle of $\h_g$ in $\A_g$ vanishes along the canonical image of $C$. Since $F_g$ vanishes along $\h_g$, $dF_g$ is a quadric containing the canonical image of $C$; in other words, it has to vanish when evaluated at any point $p$ of $C$. This concludes the proof of Theorem \ref{second_equations}.
\end{subsection}

\begin{subsection}{Other results about Niemeir quadratic forms}

With similar tools, we can prove other results about the behaviour of these modular forms on the moduli space of curves and abelian varieties.
\begin{theorem}[\cite{PhD} Corollary 11.2]
Let $P$ and $Q$ be two even positive definite unimodular quadratic forms of rank 24 with the same number of roots, then the stable modular form 
\[F:=\Theta_{P}-\Theta_{Q}\]
is zero on $\M_g$ for $g\leq 4$, and it cuts a divisor of slope 12 on $\M_5$.
\end{theorem}
\begin{theorem}[\cite{PhD} Theorem 11.3]\label{CuspForms}
The following degree 5 modular forms are non-trivial cusp forms
\begin{align*}\label{R5}
&\Theta(D_{16}E_8)-\Theta(E_8^3)-\frac{21504}{24}(\Theta(A_5^4D_5)-\Theta(D_4^6))\\
&\Theta(D_{16}E_8)-\Theta(E_8^3)-\frac{21504}{216}(\Theta(A_9^2D_6)-\Theta(D_6^4))\\
&\Theta(D_{16}E_8)-\Theta(E_8^3)-\frac{21504}{480}(\Theta(A_{11}D_7E_6)-\Theta(E_6^4))\\
&\Theta(D_{16}E_8)-\Theta(E_8^3)-\frac{21504}{-2520}(\Theta(A_{17}E_7)-\Theta(D_{10}E_7^2))
\end{align*}
where, for typographical reasons, we write $\Theta(Q)$ rather than $\Theta_{Q,5}$.
\end{theorem}
\end{subsection}

\end{section}

\bibliographystyle{alpha}


\end{document}